\documentclass{amsart}

\usepackage{amsmath}
\usepackage{amssymb}
\usepackage[all]{xy}

\def\eu{\mathfrak}
\def\ma{\mathbb}

\newcommand{\Gal}{\operatorname{Gal}}

\newcommand{\integer}{\genfrac[]{0.5pt}0}
\newcommand{\mayor}{\genfrac\lceil\rceil{0.5pt}0}
\newcommand{\mayorchico}{\genfrac\lceil\rceil{0.5pt}1}

\newcounter{bean}

\def\l{
\begin{list}
{\rm{(\alph{bean})}}{\usecounter{bean}
\setlength{\labelwidth}{0.8in}
\setlength{\labelsep}{0.3cm}
\setlength{\leftmargin}{1cm}}}

\numberwithin{equation}{section}
\theoremstyle{thmit}
\newtheorem{theorem}{Theorem}[section]
\newtheorem{proposition}[theorem]{Proposition}
\newtheorem{lemma}[theorem]{Lemma}
\newtheorem{corollary}[theorem]{Corollary}

\title[Artin--Schreier and Cyclotomic Extensions]
{Artin--Schreier and Cyclotomic Extensions}
 
\author[J.C. Salas]{Julio Cesar Salas--Torres}
\address{Academia de Matem\'aticas\\
Universidad Aut\'onoma de la Ciudad\\
de M\'exico, Plantel San Lorenzo Tezonco\\
Prolongaci\'on San Isidro No. 151\\
Col. San Lorenzo, Iztapalapa,\\
M\'exico, D.F., C.P. 09790, M\'exico}
\email{jcstorres88@hotmail.com, torres1jcesar0@gmail.com}

\author[M. Rzedowski]{Martha Rzedowski--Calder\'on}
\address{Departamento de Control Autom\'atico\\
Centro de Investigaci\'on y\\
de Estudios Avanzados del I.P.N.\\
Apartado Postal 14-740,\\
M\'exico D.F., CP 07000, M\'exico}
\email{mrzedowski@ctrl.cinvestav.mx}

\author[G. Villa]
{Gabriel Villa--Salvador}
\address{Departamento de Control Autom\'atico\\
Centro de Investigaci\'on y\\
de Estudios Avanzados del I.P.N.\\
Apartado Postal 14-740,\\
M\'exico D.F., CP 07000, M\'exico}
\email{gvillasalvador@gmail.com, gvilla@ctrl.cinvestav.mx}

\date{July 9, 2013}

\begin{document}

\begin{abstract}

From class field theory we have that any  Artin--Schreier
extension of a rational congruence function
field is contained in the composite
of cyclotomic function fields and a constant field extension.
In this paper we prove this result without using class field theory.

\end{abstract}



\maketitle

\section{Introduction}\label{S1}

Let $K$ be a function field over its constant field $k$ of characteristic $p > 0$.
The extension $L/K$ is a cyclic extension of degree $p$ if and only 
if there exists  $z \in L$ such that  $L = K(z),$ where 
\[
z^p - z = a
\] 
for some $a \in K$ with $a \notin {\wp}(K) := \{ b^p - b \mid b \in K\}$.
Such extension $L/K$ is called an {\em Artin--Schreier} extension.  

Let $K=k(T)$ be a field of rational functions where $k$ is a
perfect field of characteristic  $p>0$ and let $L/K$ be a cyclic extension
 of degree $p.$ Then $L=K(y),$ where $y$ satisfies
an equation of the form 
\begin{gather*}
y^p-y=s(T),\\
\intertext{with $s(T)\notin\wp(K)$ and the divisor of $s(T)$ given by}
(s(T))_K=\dfrac{\mathfrak{C}}{{\mathcal P}_1^{\alpha_1}\cdots
{\mathcal P}_r^{\alpha_r}},
\end{gather*}
 with $r\geq 0$, ${\mathcal P}_i$ a prime divisor in
$K$, $\alpha_i\in\ma{N},$ $(\alpha_i,p)=1,$ $\mathfrak{C}$ an integral
divisor relatively prime to ${\mathcal P}_i$ for $i\in\{1,...,r\}.$
The prime divisor ${\mathcal P}_\infty$ may or may not be included in this set of divisors.
This  equation is known as  the Artin--Schreier equation in its 
{\em normal form}.

The  ramified primes in
$L/K$ are precisely ${\mathcal P}_1,...,{\mathcal P}_r$ and are
totally and wildly ramified (see \cite{Has35}). The different of $L/K$ is
$$\eu{D}_{L/K}=\displaystyle\prod_{i=1}^r\mathfrak{p}_i^{(\alpha_i+1)(p-1)}$$
where $\mathfrak{p}_i$ is the divisor in $L$ that divides ${\mathcal
P}_i,$ for all $i\in\{1,...,r\}$ (see \cite[page 172]{Vil2006}).

Let $k={\ma F}_q(T)$ be
a congruence rational function field, where ${\ma F}_q$ is the finite
field of $q=p^s$ elements and $p$ is a prime number.
Let $R_T={\ma F}_q[T]$ be the ring of 
polynomials. 
 For $N\in R_T\setminus
\{0\}$, $\Lambda_N$ denotes the $N$--torsion of the Carlitz
module. Finally, $k(\Lambda_N)$
denotes the $N$--th cyclotomic
function field (see \cite {Hay74}). We have $G_N:=\Gal(k(\Lambda_N)
/k)\cong \big(R_T/(N)\big)^{\ast}$ and $\Phi(N):=|G_N|$. For an
irreducible polynomial $P\in R_T$ and $\alpha\in{\ma N}$ it
holds $\Phi(P^{\alpha})=q^{(\alpha-1)d}(q^d-1)$ where $d=\deg P$.

We know that the fact that
any Artin--Schreier extension of a rational congruence
function field is contained in the composite of
cyclotomic extensions and a constant field extension
is a consequence of the Kronecker--Weber Theorem
for function fields proved by David Hayes using the Reciprocity
Law of class field theory (\cite{Hay74}). 
The purpose 
of this paper is to give a combinatorial proof that any Artin--Schreier
extension of a rational congruence function field is contained in the composite
of cyclotomic function fields and a constant field
extension that are explicitly described.

The proof is given as follows. First, we consider Artin--Schreier
extensions with only one prime ramifying. We compute the number
of such extensions having their conductor a divisor of a given
power of the ramified prime. Next, we prove that this number
equals the number of Artin--Schreier extensions contained in
the composite of the cyclotomic function field generated
by the same power of the ramified prime
and a constant extension of degree $p$. The general case
follows immediately from this case using partial fractions
decomposition.

\section{The result}

The following proposition provides different manners that
an Artin--Schreier extension can be generated.

\begin{proposition} \label{pro15}
Let $K$ be a field of characteristic $p>0$ and let $L_1=K(y)$ and
$L_2=K(z)$ be cyclic extensions of degree $p$
over $K$
 given by $$y^p-y=a_1\in K \ \ and \ \  z^p-z=a_2\in K.$$ 
 Then the following statements are equivalent:
\l
\item $L_1=L_2.$
\item $z=jy+b$ for some $1\leq j\leq p-1$ and $b\in K.$
\item $a_2=ja_1+(b^p-b)$ for some $1\leq j\leq p-1$ and $b\in K.$
\end{list}
\end{proposition}

\begin{proof} See \cite[page 171]{Vil2006}.
\end{proof}

We now consider Artin--Schreier extensions where
precisely one prime divisor ramifies. First we treat
the case when the divisor corresponds to a monic
irreducible polynomial $P$. 

\begin{proposition}\label{prop1}
Let $K=\ma{F}_q(T),$ $q=p^t,$ $P$ be a monic irreducible polinomial
in $\ma{F}_q[T], d=\deg P, L_1=K(y),$
where 
\begin{gather*}
y^p-y=\dfrac{f(T)}{P^{\alpha}},\\
\intertext{$f(T) \in \ma{F}_q[T],
\deg f(T) \leq d\alpha, (f(T),P)=1, \alpha\in\ma{N},
(\alpha,p)=1$, $\alpha_0:=\left[\dfrac{\alpha}{p}
\right]\hspace{-0.1cm}$, the integer part
 of $\dfrac{\alpha}{p}$, and $L_2=K(z)$, where}
z^p-z=\dfrac{g(T)}{P^{\alpha}},
\end{gather*}
$g(T) \in \ma{F}_q[T], \deg g(T) \leq d\alpha,$
$(g(T),P)=1$.
 Then the following statements are equivalent:
\l
\item $L_1=L_2.$
\item $z=jy+c,$ where $j\in\{1,...,p-1\},$ $c=\dfrac{h(T)}
{P^{\alpha_0}}$ with $h(T)\in \ma{F}_q[T],$ where either 
$h(T) = 0$ or $\deg h(T)\leq d\alpha_0.$
\item $\dfrac{g(T)}{P^{\alpha}}=j\dfrac{f(T)}{P^{\alpha}}+c^p-c,$
where $j\in\{1,...,p-1\},$ $c=\dfrac{h(T)}{P^{\alpha_0}}$ with $h(T)\in
\ma{F}_q[T],$ where either $h(T) = 0$ or $\deg h(T)\leq d\alpha_0.$ 
\end{list}
\end{proposition}

\begin{proof} The equivalences follow from Proposition \ref{pro15}. We verify the conditions to be met by $c.$

  Let
$Q$ be a monic irreducible polinomial in $\ma{F}_q[T],$ $P\neq Q.$
We have $\nu_{Q}(c)\geq 0$ since if $\nu_{Q}(c)< 0$ we would have
\begin{align*}
0&\leq\nu_{Q}\left(\dfrac{g(T)}{P^{\alpha}}\right)=\nu_{Q}\left(j\dfrac{f(T)}{P^{\alpha}}+c^p-c\right)\\
&=\min\left\{\nu_{Q}\left(\dfrac{f(T)}{P^{\alpha}}
\right),p\nu_{Q}(c),\nu_{Q}(c)\right\}=p\nu_{Q}(c)<0,
\end{align*}
which is a contradiction. Similarly, we have
$\nu_{\infty}(c)\geq 0.$ Finally, we have $c=0$ or $\nu_{P}(c)\leq 0$ since
$\nu_{P}(c)>0$ implies
$0\geq\deg (c)=\displaystyle\sum\limits_{Q\in{\mathcal
P}_K}\nu_Q(c)>0$ and this is a contradiction.

Therefore we consider the following possibilities:

\l
\item $c=0.$\vspace{2mm}

\item $c\neq 0,$ $\nu_P(c)=0$ which implies $c\in\ma{F}_q^*.$\vspace{2mm}

\item $\nu_P(c)<0,$ let $\gamma:=-\nu_P(c),$ then 
$\nu_P(c^p-c)=\min\{p\nu_P(c),\nu_P(c)\}$
$=-p\gamma\neq -\alpha,$ because $(\alpha,p)=1.$ Thus
\begin{align*}
-\alpha&=\nu_{P}\left(\dfrac{g(T)}{P^{\alpha}}\right) =
\nu_{P}\left(j\dfrac{f(T)}{P^{\alpha}}+c^p-c\right)\\
&=\min\left\{\nu_{P}\left(j\dfrac{f(T)}{P^{\alpha}}\right),\nu_{P}(c^p-c)\right\}
=\min\{-\alpha,-p\gamma\}.
\end{align*}
This implies $-\alpha< -p\gamma$, so that
 $p\gamma<\alpha$. Therefore $\gamma<\dfrac{\alpha}{p}.$ Hence
$\gamma\leq\alpha_0=\left[\dfrac{\alpha}{p}\right]\hspace{-0.1cm},$
then $c=\dfrac{h_1(T)}{P^{\gamma}}$ with $h_1(T)\in
{\ma F}_q[T]$, $(h_1(T),P)=1,$
$\deg h_1(T)\leq d\gamma$ and $\gamma\leq \alpha_0$,
which implies $c=\dfrac{h_1(T)P^{\alpha_0-\gamma}}{P^{\alpha_0}}.$
\end{list}
In the first two cases we put $h(T) = cP^{\alpha_0}$ and in the third $h(T)=h_1(T)P^{\alpha_0-\gamma}.$

We conclude that $c=\dfrac{h(T)}{P^{\alpha_0}}$ with either $h(T) = 0$ or $\deg
h(T)\leq d\alpha_0.$ 
\end{proof}

\begin{corollary}\label{corol4}
Let $K$, $q$, $P$ and $d$ be as in Proposition \ref{prop1}. Let $L=K(y)$,
be an Artin--Schreier extension given in normal form, that is
\[
y^ p-y=\dfrac{f(T)}{P^{\alpha}},
\]
where $f(T) \in \ma{F}_q[T],
\deg f(T) \leq d\alpha,$ $(f(T),P)=1,$ $\alpha\in\ma{N}$,
$(\alpha,p)=1$. Let
$\alpha_0=\integer {\alpha}{p}$. Then there are
$(p-1)q^{(d\alpha_0+1)}$ elements $z\in L$ such that $L=K(z)$
is also given in normal form
\begin{gather}\label{Eq0}
z^ p-z=\dfrac{g(T)}{P^{\alpha}}.
\end{gather}

Further, these elements $z$
are given by $z=jy+c,$ where $1\leq j\leq p-1$ and
$c=\dfrac{h(T)}{P^{\alpha_0}}$, with $h(T) \in \ma{F}_q[T]$,
$\deg h(T)\leq d\alpha_0.$ 

Moreover, there are
\[
\dfrac{p-1}{p}q^{(d\alpha_0+1)}
\]
different equations of the form (\ref {Eq0})
 with $g(T) \in \ma{F}_q[T],$
$\deg g(T)\leq d\alpha$, $(g(T),P)=1$, so that $z$ is as above and
$\dfrac{g(T)}{P^{\alpha}}=j\dfrac{f(T)}{P^{\alpha}}+c^p-c.$
\end{corollary}

\begin{proof}
By Proposition \ref{prop1} the number of possibilities
 for $j$ is $p-1$ and the number of possible $h(T)$ is
 $q^{(d\alpha_0+1)}$, thus the number of possible $z$
 is $(p-1)q^{(d\alpha_0+1)}$. To obtain the number of
 possible equations just note that $c^p-c=b^p-b$
 if and only if $(c-b)^p=c-b$ if and only if $c-b\in\ma{F}_p.$ 
\end{proof}

\begin{corollary}\label{corol5}
 Let $K$, $q$, $P$ and $d$ be as above.
 Then the number of different Artin--Schreier  extensions 
$L=K(y)$ where 
\[
y^p-y=\dfrac{f(T)}{P^{\alpha}}
\]
with $f(T) \in \ma{F}_q[T], \deg
 f(T) \leq d\alpha,$ $(f(T),P)=1,$ $\alpha\in\ma{N},$
$(\alpha,p)=1,$ is
\[
N_{\alpha}:=\dfrac{p}{p-1}\Phi\left(P^{\alpha-\alpha_0}\right),
\]
where $\alpha_0=\integer {\alpha}{p}$.
\end{corollary}

\begin{proof}
By the division algorithm we have
$f(T)=aP^{\alpha}+h(T),$ where $a\in\ma{F}_q,$ $h(T)\in
\ma{F}_q[T],$ with either $\deg h(T)\leq d\alpha-1$ and $(h(T),P)=1$ or $h(T) = 0$.
Then $y^p-y=a+\dfrac{h(T)}{P^{\alpha}}.$ The number of equations of this type is 
\begin{gather*}
q\Phi(P^{\alpha})=q\cdot q^{(\alpha-1)d}(q^d-1).
\intertext{By Corollary \ref{corol4} we obtain that there are}
N_{\alpha}=\dfrac{q\cdot
q^{(\alpha-1)d}(q^d-1)}{\dfrac{p-1}{p}q^{(d\alpha_0+1)}}
=\dfrac{p}{p-1}q^{(\alpha-\alpha_0-1)d}(q^d-1)=
\dfrac{p}{p-1}\Phi(P^{\alpha-\alpha_0})
\end{gather*}
different cyclic extensions $L$ of degree $p$ over $K$ where $P$ is the only prime ramified and the  power $\alpha$ appears
 in the  Artin--Schreier equation in its normal form. 
 \end{proof}
 
 \begin{lemma}\label{lextra}
 Let $K={\ma F}_q(T)$, $q=p^t$, $P$ a monic irreducible polynomial
 in ${\ma F}_q[T]$, $d=\deg P$ and $\alpha\in {\ma N}$, $(
 \alpha,p)=1$.
 \l
 \item
 Let $F=K(y)$ where
 \[
 y^p-y=\frac{f(T)}{P^{\alpha}}
 \]
 with $f(T)\in {\ma F}_q[T]$, $\deg f(T)\leq \alpha d$ and $(f(T),P)=1$.
 
 Assume $F\subseteq K(\Lambda_M)$ for some $M\in R_T
 \setminus\{0\}$. Then $F\subseteq K(\Lambda_{P^{\alpha+1}})$
 but $F\nsubseteq K(\Lambda_{P^{\alpha}})$.
 
 \item Conversely, if $F/K$ is an Artin--Schreier extension such that
$F\subseteq K(\Lambda_{P^{\alpha+1}})$ but $F\nsubseteq K(
\Lambda_{P^{\alpha}})$, then $F=K(y)$, where
\[
y^p-y=\frac{f(T)}{P^{\alpha}}
\]
where $f(T)\in {\ma F}_q[T]$, $\deg f(T)\leq \alpha d$ and $(f(T),P)=1$.
\end{list}
\end{lemma}
 
 \begin{proof}
(a)   We have
that $\mathcal {P}$, the prime divisor associated
to the polynomial $P$, is the only ramified prime divisor in $F/K$ and 
the different of $F/K$ is
$$\mathfrak{D}_{F/K}=\mathfrak{p}^{(\alpha+1)(p-1)},$$ where
$\mathfrak{p}$ is the only  prime divisor in $F$ that divides ${\mathcal P}$. 
Then the discriminant of $F/K$ is

$$\delta_{F/K}={\mathcal P}^{(\alpha+1)(p-1)}.$$

Since $F$ is contained in a cyclotomic field  and the
extension $F/K$ is cyclic, $F$ is the field associated to some character $\Theta$ of order $p$ and conductor ${\eu F}_{\Theta}$
(see ~\cite[Chapter 12]{Vil2006}).

By the conductor-discriminant formula (see ~\cite{RzVi2011}) we have
\[
\delta_{F/K}=\displaystyle\prod\limits_{\varphi\in\langle\Theta\rangle}
{\eu F}_{\varphi}={\eu F}_{\Theta^0}{\eu F}_{\Theta^1}\cdots {\eu F}_{\Theta^{p-1}}
={\eu F}_{\Theta}^{p-1}.
\]

 Then ${\mathcal P}^{(\alpha+1)(p-1)}={\eu F}_{\Theta}^{p-1},$ so that 
 ${\eu F}_{\Theta}={\mathcal P}^{(\alpha+1)}.$
   Hence we conclude that
  $F\subseteq K(\Lambda_{P^{\alpha+1}})$ but $F\nsubseteq K(\Lambda_{P^{\alpha}})$ (see~\cite[Chapter 12]{Vil2006}).

(b) Since $F\subseteq K(\Lambda_{P^{\alpha+1}})$, $P$ is the only
ramified prime in $F/K$. Then $F=K(y)$ where $y^p-y=\frac{f(T)}{P^{\beta}}$,
where $f(T)\in {\ma F}_q[T]$, $\deg f(T)\leq \beta d$, $(f(T),P)=1$.
It follows from (a) that $\beta=\alpha$.
 \end{proof}

\begin{proposition}\label{prop2}
Let $K=\ma{F}_q(T),$ $q=p^t$ and $P$ be a  monic
irreducible polynomial in $\ma{F}_q[T],$ $d=\:\deg P,$
$\beta\in\ma{N},$ $\beta\geq 2.$ Then the number of Artin--Schreier extensions
contained in the cyclotomic field
$K(\Lambda_{P^{\beta}})$ is
\[
\mathcal{N}_{\beta}:=\dfrac{q^{(\beta-\mayorchico{\beta}{p})d}-1}{p-1},
\]
where $\mayor{\beta}{p}$ denotes the ceiling of
$\dfrac{\beta}{p},$ namely the minimum integer greater than or equal to
$\dfrac{\beta}{p}.$
\end{proposition}

\begin{proof} Since the lattice of subgroups of an abelian 
group is symmetric, we have by Galois theory
$\mathcal{N}_{\beta}$ equals the number of subgroups 
of order $p$ of the Galois group $\left(R_T/P^{\beta}\right)^*$ of the extension
$K(\Lambda_{P^{\beta}})/K.$ Let $r_p$ be the number 
of elements of order $p$ in $\left(R_T/P^{\beta}\right)^*\hspace{-0.2cm}.$
We consider the exact sequence
\[
\begin{array}{ccccccccc}
1&\longrightarrow&D_{P^{\beta},P}&\longrightarrow& \left(R_T/P^{\beta}\right)^{\ast}
&\longrightarrow& \left(R_T/P\right)^{\ast}&\longrightarrow&1.\\
&&&&A\bmod P^{\beta}&\longmapsto&A\bmod P
\end{array}
\]

Therefore
\[
\left(R_T/P^{\beta}\right)^*\cong
D_{P^{\beta},P}\times\left(R_T/P\right)^{\ast}.
\]

We have
   $D_{P^{\beta},P}=\{A\bmod P^{\beta}|A\equiv 1\bmod P\}=\{A\bmod P^{\beta}|A=1+h(T)P,$
   where $h(T)\in \ma{F}_q[T]$ and $\deg h(T)\leq(\beta-1)d-1\}.$

Observe that $|D_{P^{\beta},P}|=q^{(\beta-1)d}$.
We also note that $A\bmod P^{\beta}\equiv 1\bmod P^{\beta}$ if and only if
$h(T)=0.$ Then $A\bmod P^{\beta}$ is of order $p$ if and only if
$h(T)\neq 0$ and $(1+h(T)P)^p=1+h(T)^pP^p\equiv 1\bmod P^{\beta}$
 if and only if $P^{\beta}|h(T)^pP^p.$ We have two cases:

\l
\item If $\beta\leq p,$ it holds for all $A\bmod\left(R_T/P^{\beta}\right)^*\in D_{P^{\beta},P}$ and so $r_p=q^{(\beta-1)d}-1.$
\item If $\beta>p,$ we have $P^{\beta}g(T)=h(T)^pP^p,$
for some $g(T)\in {\ma F}_q[T]$. Therefore
 $P^{\beta-p}g(T)
 =h(T)^p.$ So $h(T)=P^{\gamma}h_1(T)$ for some $\gamma\in\ma{N}$ and
 $h_1(T)\in{\ma F}_q[T]$ with $(h_1(T),P)=1.$
We have $P^{\beta-p}g(T)=
P^{\gamma p}h_1(T)^p.$ Then $\gamma p\geq
\beta-p,$  $\gamma\displaystyle \geq\frac{\beta}{p}-1,$
$\gamma+1\displaystyle \geq\frac{\beta}{p}.$ Thus,
$\gamma+1\displaystyle \geq\mayor{\beta}{p}$.
Therefore
$h(T)=P^{(\mayorchico{\beta}{p}-1)}h_2(T),$ where
\begin{align*}
\deg h_2(T)&=\deg
h(T)-\left(\mayor{\beta}{p}-1\right)d\leq(\beta-1)d-1-
\left(\mayor{\beta}{p}-1\right)d\\
&=\left(\beta-\mayor{\beta}{p}\right)d-1.
\end{align*}
It follows that
$r_p=q^{\left(\beta-\mayorchico{\beta}{p}\right)d}-1$
is the number of elements of order $p$ of
$\left(R_T/P^{\beta}\right)^*\hspace{-0.2cm}.$

Then, in any case, the number of subgroups of order $p$ of
$\left(R_T/P^{\beta}\right)^* $ is
\begin{gather*}
\mathcal{N}_{\beta}=\dfrac{r_p}{p-1}=\frac{q^{\left(\beta-
\mayorchico{\beta}{p}\right)d}-1}{p-1}.
\end{gather*}
\end{list}
\end{proof}

\begin{corollary}\label{corol6}
Let $K=\ma{F}_q(T),$ $q=p^t$ and $P$ be a monic
irreducible  polynomial  in $\ma{F}_q[T],$ $d=\deg P,$
$\alpha\in\ma{N},$ $(\alpha,p)=1,$
$\alpha_0=\left[\dfrac{\alpha}{p}\right].$ Then the number of Artin--Schreier 
extensions contained in
$K(\Lambda_{P^{\alpha+1}})$ but not in $K(\Lambda_{P^{\alpha}})$ is
$$\dfrac{\Phi(P^{\alpha-\alpha_0})}{p-1}.$$
\end{corollary}

\begin{proof} By Proposition \ref{prop2} we have that number of Artin--Schreier 
extensions contained in
$K(\Lambda_{P^{\alpha+1}})$ but not in $K(\Lambda_{P^{\alpha}})$ is
 equal to 
 \begin{align*}
 \mathcal{N}_{\alpha + 1} - \mathcal{N}_{\alpha}&=\dfrac{q^{\left(\alpha+1-\mayorchico
{\alpha+1}{p}\right)d}-1}{p-1}-\dfrac{q^{\left(\alpha-\mayorchico{\alpha}{p}\right)d}-1}{p-1}\\
&=\frac{q^{\left(\alpha+1-\mayorchico{\alpha+1}{p}\right)d}-q^{\left(\alpha-
\mayorchico{\alpha}{p}\right)d}}{p-1}=
\dfrac{q^{\alpha
d}\left(q^{\left(1-\mayorchico{\alpha+1}{p}\right)d}-q^{-\mayorchico{\alpha}{p}d}\right)}{p-1}\\
&=\dfrac{q^{\alpha
d}q^{-\mayorchico{\alpha}{p}
d}(q^d-1)}{p-1}=\dfrac{q^{\left(\alpha-\mayorchico{\alpha}{p}\right)d}(q^d-1)}{p-1}=
\dfrac{q^{\left(\alpha-\left[\frac{\alpha}{p}\right]-1\right)d}(q^d-1)}{p-1}\\
&=\dfrac{\Phi\left(P^{\alpha-\alpha_0}\right)}{p-1},
\end{align*}
 because when
$p\nmid\alpha,$ we have $\mayor{\alpha+1}{p}
=\mayor{\alpha}{p}$ and
$\mayor{\alpha}{p}=\integer{\alpha}{p}+1.$
\end{proof}

\begin{corollary}\label{corol7}
Let $K=\ma{F}_q(T),$ $q=p^t$ and $P$ be a monic
irreducible  polynomial  in $\ma{F}_q[T],$ $d=\:\deg P,$
$\alpha\in\ma{N},$ $(\alpha,p)=1,$
$\alpha_0=\left[\dfrac{\alpha}{p}\right].$  Consider the Artin--Schreier 
extensions $F=K(y)$ of $K$ where
\begin{equation}\label{Eq0'}
y^p-y=\dfrac{f(T)}{P^{\alpha}}
\end{equation}
with $f(T)\in \ma{F}_q[T],$ $\deg
f(T)\leq\alpha d$ and $(f(T),P)=1$. Then there are at least
\begin{equation*}
N_{\alpha}=\dfrac{p}{p-1}\Phi(P^{\alpha-\alpha_0})
\end{equation*}
extensions $F$ of the type
described in {\rm{(\ref {Eq0'})}} contained in $K(\Lambda_{P^{\alpha + 1}})\ma{F}_{q^p}.$
\end{corollary}

\begin{proof} Consider the following diagram. 
$$\xymatrix{K(\Lambda_{P^{\alpha+1}}) \ar @{-} [d]\ar @{-}[r] & K(\Lambda_{P^{\alpha+1}})\ma{F}_{q^p} \ar @{-} [d] \\
 F \ar @{-} [d]_{p} \ar @{-} [r] & F\ma{F}_{q^p} \ar @{-} [d] \\
K \ar  @{-} [r]_{p} & K\ma{F}_{q^p}}$$

We have $\ma{F}_{q^p}=\ma{F}_{q}(\xi),$ where $\xi^p-\xi=\rho$ with $\rho\in\ma{F}_{q}\backslash\wp(\ma{F}_{q}).$

By Corollary \ref{corol6} there are $\frac{\Phi(P^{\alpha-\alpha_0})}{p-1}$ Artin--Schreier
extensions $F/K$ contained in $K(\Lambda_{P^{\alpha+1}})$ but not
in $K(\Lambda_{P^{\alpha}})$. By Lemma \ref{lextra} such $F$
is of type (\ref{Eq0'}). From one such $F$ we 
obtain $p$ fields $F_i=K(y_i),$ where
$y_i=y+i\xi$ for $0\leq i\leq p-1,$ of degree $p$ over $K$ and
contained in $K(\Lambda_{P^{\alpha+1}})\ma{F}_{q^p}.$ We will verify
that these fields are of the type required, that they are different 
and also that if
$E,F\subseteq K(\Lambda_{P^{\alpha+1}})$ are of the type required and
$E\neq F,$ then $E_i\neq F_j,$ for all $i,j\in\{0,...,p-1\}.$
With the above we conclude that at least
$\dfrac{p}{p-1}\Phi(P^{\alpha-\alpha_0})$ extensions $F$ of the type required are
contained in the composite of the cyclotomic extension and the constant extension.

We have
\begin{align*}
(y+i\xi)^p-(y+i\xi)&=y^p-y+i^p\xi^p-i\xi=y^p-y+i(\xi^p-\xi)\\
&=\dfrac{f(T)}
{P^{\alpha}}+i\rho=\dfrac{f(T)+i\rho
P^{\alpha}}{P^{\alpha}}=\dfrac{g_i(T)}{P^{\alpha}}
\end{align*}
with $g_i(T)\in
\ma{F}_q[T],$ $\deg g_i(T)\leq\alpha$ for all
$i\in\{0,...,p-1\}.$

Then $F_i=K(y+i\xi)$ is an extension of the type required.
Suppose now that $0\leq i,j\leq p-1,$ $i\neq j$ and $F_i=F_j.$
Then $y+i\xi,$ $y+j\xi\in F_i=F_j,$ then $(i-j)\xi\in F_i=F_j,$
therefore $\xi\in F_i=F_j,$ so that $y\in F_i=F_j=F.$ Thus
$\xi\in F_i=F_j=F,$ which is a contradiction because the constant field of $F$ is $\ma{F}_q.$

Finally, assume that $F,E$ are extensions of the type required,
contained in $K(\Lambda_{P^{\alpha+1}})$ with $F\neq E$ and $E_i=F_j$
for some $i,j\in\{0,...,p-1\}.$

Say $F=K(y)$ with $y^p-y=\dfrac{f(T)}{P^{\alpha}},$ $f(T)\in
\ma{F}_q[T],$ $\deg f(T)\leq\alpha$ and $E=K(z)$ with
$z^p-z=\dfrac{g(T)}{P^{\alpha}},$ $g(T)\in \ma{F}_q[T],$
$\deg g(T)=\alpha.$

We have $i\neq 0$ and $j\neq 0,$ since otherwise $\xi\in
K(\Lambda_{P^{\alpha+1}}),$  which is a contradiction because the constant field of $K(\Lambda_{P^{\alpha+1}})$ is $\ma{F}_q.$
Let $l=i^{-1}j.$ We have $ly+j\xi=l(y+i\xi),$ $z+j\xi\in F_i=E_j.$
Then $z-ly\in F_i=E_j.$ But $z-ly\not\in K,$ for $F\neq E.$
Therefore $\ma{F}_{q^p}\subseteq F_i=E_j=K(z-ly)\subseteq
K(\Lambda_{P^{\alpha+1}}),$ which is a contradiction. 
\end{proof}

Next lemma is the main step to our main result. We prove that any cyclic extension of degree
$p$
over $K$ in which $\mathcal P$ is the only ramified prime divisor and it appears to the power $\alpha$ in the Artin-Schreier equation in its normal form is contained in the composite of a cyclotomic field and a constant extension.

\begin{lemma}\label{lem5}
Let $p$ be a prime number, $q=p^t$ and $P$ be a monic
irreducible polynomial  in  $\ma{F}_q[T],$ $d=\:\deg P.$ Let
$K=\ma{F}_q(T)$ and $F=K(y),$ where
$$y^p-y=\displaystyle\frac{f(T)}{P^{\alpha}}$$ with $f(T)\in
\ma{F}_q[T],$ $\deg f(T)\leq\alpha d$, $(f(T),P)=1$, $\alpha\in\ma{N},$ and
$(\alpha, p)=1.$ Then $$F\subseteq
K(\Lambda_{P^{\alpha+1}})\ma{F}_{q^p}.$$
\end{lemma}

\begin{proof}  By Corollaries \ref{corol5} and \ref{corol7} we have that the number of cyclic extensions of degree $p$
over $K$ in which $P$ is the only ramified prime and it appears to the power $\alpha$ in the Artin-Schreier equation in its normal form must be equal to the number of such extensions contained in the composite of the cyclotomic and the constant extensions mentioned above. Therefore we conclude that any such extension $F$ is contained in the composite $K(\Lambda_{P^{\alpha+1}})\ma{F}_{q^p}$. 
\end{proof}

In the following result are considered Artin--Schreier extensions in which the infinite prime divisor is the only ramified prime divisor.

\begin{lemma} \label{lem6}
Let $p$ be a prime number and $q=p^t.$  Let $K=\ma{F}_q(T)$ and
$F=K(y),$ where $y^p-y=f(T)\in \ma{F}_q[T],$ $\deg f(T) =
\alpha,$ $\alpha\in\ma{N},$ and $(\alpha, p)=1.$ Then $$F\subseteq
K(\Lambda_{\frac{1}{T^{\alpha+1}}})\ma{F}_{q^p}.$$
\end{lemma}

\begin{proof} It follows from Lemma \ref{lem5} if we observe that
$K(T)=K\left(\dfrac{1}{T}\right)$.
\end{proof}

We note that to simplify the statement of the main result in the abstract and in the introduction, we have included the field $K(\Lambda_{\frac{1}{T^{\alpha+1}}})$ among cyclotomic fields. 

Finally, we present our main result. 

\begin{theorem}\label{prop25} Let $K=\ma{F}_q(T)$, where $q=p^t$ and $p$ is a prime number. Let $F/K$ be an Artin--Schreier extension, namely $F=K(y),$ where
$$y^p-y=s(T),$$ with $s(T)\in K,$
$s(T)\not\in\wp(K),$
$$(s(T))_K=\dfrac{\mathfrak{C}}{{\mathcal P}_1^{\alpha_1}\cdots
{\mathcal P}_r^{\alpha_r}},$$ 
where ${\mathcal P}_i$ is a prime divisor in
$K,$ $\alpha_i\in\ma{N},$ $(\alpha_i,p)=1,$ $\mathfrak{C}$ is an integral divisor relatively prime to ${\mathcal P}_i$ for $i\in\{1,...,r\}.$
\l
\item If the prime divisor $\mathcal P_{\infty}$ is not ramified in $F/K,$ that is, if $\mathcal P_{\infty}$ is not a factor of
 $\displaystyle\prod\limits_{i=1}^r{\mathcal P}_i^{\alpha_i},$ then 
$F\subseteq
K(\Lambda_{\prod\limits_{i=1}^rP_i^{\alpha_i+1}})\ma{F}_{q^p},$
\item If ${\mathcal P}_{\infty}$ is ramified in $F/K,$ that is, if ${\mathcal P}_{\infty}$ is a factor of $\displaystyle\prod\limits_{i=1}^r{\mathcal P}_i^{\alpha_i},$ say  ${\mathcal P}_{\infty} = {\mathcal P}_1$, then $F\subseteq
K(\Lambda_{\frac{1}{T^{\alpha_1+1}}})K(\Lambda_{\prod\limits_{i=2}^rP_i^{\alpha_i+1}})\ma{F}_{q^p}.$
\end{list}
\end{theorem}

\begin{proof} The divisor ${\mathcal P}_{i}$ corresponds to the polynomial $P_i$, for $i \in \{1, \ldots r\}$ in case (a) and for $i \in \{2, \ldots r\}$ in case (b). 
\l
\item By the partial fractions method we have:
$$s(T) = \frac{f(T)}{P_1^{\alpha_1}\cdots P_r^{\alpha_r}}=\frac{f_1(T)}{P_1^{\alpha_1}}+\cdots+\frac{f_r(T)}{P_r^{\alpha_r}},$$
where $f_i(T)\in \ma{F}_q[T]$ and $\deg f_i(T)\leq\alpha_i\:
\deg P_i$ for all $i\in\{1,...,r\}.$ Consider
$F_i=K(y_i),$ where
$y_i^p-y_i=\displaystyle\frac{f_i(T)}{P_i^{\alpha_i}}$ for
$i\in\{1,...,r\}.$  By Lemma \ref{lem5}, it follows that
$F_i\subseteq K(\Lambda_{P_i^{\alpha_i+1}})\ma{F}_{q^p}$ for
$i\in\{1,...,r\}.$ We note that we can put $y=y_1+\cdots+y_r$
since $y^p-y=y_1^p-y_1+\cdots+y_r^p-y_r=\dfrac{f(T)}{P_1^{\alpha_1}\cdots
P_r^{\alpha_r}}.$ Then $F=K(y)\subseteq
K(\Lambda_{P_1^{\alpha_1+1}})\cdots
K(\Lambda_{P_r^{\alpha_r+1}})\ma{F}_{q^p}=K(\Lambda_{P_1^{\alpha_1+1}\cdots P_r^{\alpha_r+1}})\ma{F}_{q^p}.$

\item In this case ${\mathcal P}_{\infty}={\mathcal P}_1$ and we have $$y^p-y=s(T)=
\dfrac{g(T)}{P_2^{\alpha_{2}}\cdots P_r^{\alpha_r}}$$
with $\deg\left(\dfrac{g(T)}{P_2^{\alpha_{2}}\cdots
P_r^{\alpha_r}}\right)=\alpha_1\in\ma{N},$ $g(T)\in \ma{F}_q[T]$ and
$P_2,...,P_r$ monic and irreducible polynomials. By the division algorithm $$g(T)=(P_2^{\alpha_{2}}\cdots
P_r^{\alpha_r})h(T)+l(T),$$ 
 where $h(T),\: l(T)\in \ma{F}_q[T]$, 
$\deg h(T)=\alpha_1$ and either
$\deg l(T)<\displaystyle\sum\limits_{i=2}^{r}\alpha_i\deg
P_i$ or $l(T) = 0$. Then
\[
y^p-y=h(T)+\dfrac{l(T)}{\prod\limits_{i=2}^rP_i^{\alpha_i}}.
\]
We consider $F_1=K(y_1),$ where $y_1^p-y_1=h(T).$ As in case
(a), by the partial fractions method we have:
\[
\frac{l(T)}{\prod\limits_{i=2}^rP_i^{\alpha_i}}
=\frac{l_2(T)}{P_2^{\alpha_2}}+\cdots+\frac{l_r(T)}{P_r^{\alpha_r}},
\]
where $l_i(T)\in \ma{F}_q[T]$ and $\deg
l_i(T)\leq\alpha_id_i$ for all $i\in\{2,...,r\}.$ From Lemmas
\ref{lem5} and \ref{lem6} we have
\begin{gather*}
F_1\subseteq K(\Lambda_{\frac{1}{T^{\alpha_1+1}}})\ma{F}_{q^p},\\
F_i\subseteq K(\Lambda_{\frac{1}{P_i^{\alpha_i+1}}})
\ma{F}_{q^p}\:\mathrm{for\: all}\: i\in\{2,...,r\}.
\end{gather*} 
We have
$y=y_1+\cdots+y_r$ since
$y^p-y=y_1^p-y_1+\cdots+y_r^p-y_r=h(T)+\dfrac{l(T)}{\prod\limits_{i=2}^rP_i^{\alpha_i}}.$

We conclude 
\begin{align*}
F=K(y)&\subseteq
K(\Lambda_{\frac{1}{T^{\alpha_1+1}}})K(\Lambda_{P_2^{\alpha_2+1}})\cdots
K(\Lambda_{P_r^{\alpha_r+1}})\ma{F}_{q^p}\\
&=K(\Lambda_{\frac{1}{T^{\alpha_1+1}}})K(\Lambda_{P_2^{\alpha_2+1}\cdots
P_r^{\alpha_r+1}})\ma{F}_{q^p}
=K(\Lambda_{\frac{1}{T^{\alpha_1+1}}})K(\Lambda_{\prod\limits_{i=2}^rP_i^{\alpha_i+1}})\ma{F}_{q^p}.
\end{align*}
\end{list}
\end{proof}

\bibliographystyle{line}
\bibliography{JAMS-paper}

\end{document}